\documentclass[10pt,reqno]{amsart}
\usepackage{mathtools}
\usepackage{enumitem}
\mathtoolsset{showonlyrefs}
\usepackage{graphicx}
\usepackage{indentfirst,csquotes}
\usepackage{amssymb,amsmath}
\usepackage{xcolor,hyperref,fancyhdr,etoolbox}
\usepackage{tikz}
\usetikzlibrary{positioning, shapes.geometric}
\newtheorem{theorem}{Theorem}
\newtheorem{definition}[theorem]{Definition}
\newtheorem{lemma}[theorem]{Lemma}
\newtheorem{corollary}[theorem]{Corollary}
\newtheorem{question}[theorem]{Question}
\newtheorem{remark}[theorem]{Remark}
\numberwithin{equation}{section}
\hypersetup{ colorlinks=true, linkcolor=black, filecolor=black, urlcolor=black }

\begin{document}

\title{Thin set theorem for arbitrarily many colors implies bounding}

\author{Yamato Miyata$^{1}$}
\address{$^{1}$Mathematical Institute, Tohoku University}
\email{miyata.yamato.t8@dc.tohoku.ac.jp}

\author{Keita Yokoyama$^{2}$}
\address{$^{2}$Mathematical Institute, Tohoku University}
\email{keita.yokoyama.c2@tohoku.ac.jp}

\date{\today}

\begin{abstract}
The thin set theorem $\mathsf{RT}_{<\infty,\ell}^{n}$ asserts that for every natural number $k$, each coloring $c\colon[\mathbb{N}]^n \to \{0,1,\dots,k-1\}$ admits an infinite set $H$ such that $|c([H]^n)| \le \ell$. Within the framework of the reverse mathematics of second-order arithmetic, $\mathsf{RT}_{<\infty,\ell}^{n}$ implies the $\Sigma_{n+1}^{0}$-bounding principle ($\mathsf{B}\Sigma_{n+1}^{0}$) over $\mathsf{RCA}_0$ for all natural numbers $n, \ell \ge 1$.
\end{abstract}
\maketitle
\tableofcontents

\section{Introduction}
Ramsey's theorem for $n$-tuples and any $k$ colors, denoted by $\mathsf{RT}_{<\infty}^{n}$, states that every coloring $c\colon[\mathbb{N}]^n \to \{0,1,\dots,k-1\}$ admits an infinite set $H$ such that $c([H]^n) = \{i\}$ for some $i < k$.

In the context of reverse mathematics, $\mathsf{RT}_{<\infty}^{2}$ is a prime example of a theorem that is not equivalent to any of the core systems known as the ``\textbf{Big Five}'' \cite{Hirst,Liu}. Furthermore, Ramsey's theorem also exhibits unusual behavior regarding induction. For example, Hirst \cite[Theorem 6.11]{Hirst} found that $\mathsf{RT}_{<\infty}^{2}$ derives $\mathsf{B}\Sigma_{3}^{0}$ over $\mathsf{RCA}_0$, and Slaman and the second author \cite{Slaman-Yokoyama} found that $\mathsf{RT}_{<\infty}^{2}$ does not derive $\mathsf{I}\Sigma_{3}^{0}$.
Moreover, Patey and the second author \cite{Patey-Yokoyama} showed that $\mathsf{RCA}_0+\mathsf{RT}_{2}^{2}$ is $\Pi_{3}^{0}$-conservative over $\mathsf{RCA}_0+\mathsf{B}\Sigma_{2}^{0}$. (Note that $\mathsf{RT}_{2}^{2}$ states that every computable coloring $c\colon[\mathbb{N}]^2\to\{0,1\}$ admits an infinite set $H$ such that $c([H]^2)=\{i\}$ for some $i<2$.) This result was recently improved to $\Pi_{4}^{0}$-conservativity by Le Hou\'erou, Patey and the second author \cite{Houerou-Patey-Yokoyama}.
Thus, its relationships with the \textbf{Big Five} and with induction principles have been investigated in considerable detail. Refer to Hirschfeldt \cite{Hirschfeldt} for standard topics regarding the reverse mathematics of Ramsey's theorem.

Historically, the original thin set theorem asserts that for any coloring $c \colon [\mathbb{N}]^n \to \mathbb{N}$, there exists an infinite set $H$ such that $c([H]^n) \neq \mathbb{N}$ (the image avoids at least one color). By convention, however, we usually refer to the following stronger statement as the thin set theorem:

\begin{definition}
For any natural numbers $n, \ell \ge 1$, $\mathsf{RT}_{<\infty,\ell}^{n}$ is the statement that for every $k$ and every coloring $c \colon [\mathbb{N}]^n \to \{0,1,\dots,k-1\}$, there exists an infinite set $H \subseteq \mathbb{N}$ such that $|c([H]^n)| \le \ell$. For a coloring $c$, any infinite set $H$ satisfying this condition is called an (infinite) $\ell$-homogeneous set for $c$.
\end{definition}

Wang \cite{Wang} investigated the computability-theoretic aspects of this principle. Recently, Cholak and Patey \cite{Cholak-Patey} proved a striking result: for all natural numbers $n \ge 3$, $\mathsf{RCA}_0 \vdash \mathsf{ACA}_0 \leftrightarrow \mathsf{RT}_{<\infty,\ell}^{n}$ if and only if $\ell < d_{n-1}$, where $(d_n)_{n \in \omega}$ denotes the Catalan numbers defined by $d_0 = 1$ and $d_{n+1} = \sum_{i=0}^{n} d_i \cdot d_{n-i}$. This result is based on a consideration of the combinatorial and computational constraints of color avoidance.

Although the computability-theoretic aspects of the thin set theorem and its relationship with the \textbf{Big Five} have been actively investigated, its connection to mathematical induction remains largely unexplored. In this paper, we analyze and evaluate the inductive strength of the thin set theorem.
Our main result is to establish the lower bound on its inductive strength over $\mathsf{RCA}_0$. Specifically, we show that for all natural numbers $\ell\ge 1$, $\mathsf{RT}_{<\infty,\ell}^{n}$ derives $\mathsf{B}\Sigma_{n+1}^{0}$ over $\mathsf{RCA}_0$.

\section{Inductive strength of the thin set theorem}
In the following discussion, non-essential parameters within formulas are omitted. Here, we assume that finite sets and the basic operations $\in,\subseteq$ on them are encoded using natural numbers.
\begin{definition}\label{def:2}
For any natural numbers $n,\ell \ge 1$, we define the following principles (Note that (3) and (4) are finitely axiomatizable):

\begin{enumerate}[topsep=0.5em,itemsep=0.25em]
\item We say that a (total) $\Sigma_{n}^{0}$ formula $\varphi(x,y)$ defines a total function (we denote $\varphi\colon\mathbb{N}\to\mathbb{N}$), when $\varphi$ is expressed by a $\Sigma_{n}^{0}$ formula and satisfies $\forall x \, \exists! y \, \varphi(x,y)$. In particular, when $k$ is fixed, if $\forall x \, \exists! y < k \, \varphi(x,y)$, we denote $\varphi \colon \mathbb{N}\to \{0,1,\dots,k-1\}$.

\item For any $k$ and any total $\Sigma_{n}^{0}$-function $\varphi\colon\mathbb{N}\to \{0,1,\dots,k-1\}$, we say that a set $H$ is an (infinite) $\ell$-homogeneous set for $\varphi$ if $H$ is an infinite set and it satisfies the following:
\[ \exists s\subseteq\{0,1,\dots,k-1\} \, ( |s|\le \ell \land \forall x \, \exists y\in H \, (x<y) \land \forall h\in H \, \exists i\in s \, \varphi(h,i) ). \]

\item $\Sigma_{n}^{0}\text{-}\mathsf{RT}_{<\infty,\ell}^{1}$ is the assertion: For any $k$ and any total $\Sigma_{n}^{0}$-function $\varphi\colon\mathbb{N}\to\{0,1,\dots,k-1\}$, $\varphi$ has an $\ell$-homogeneous set.
  
In particular, $\Sigma_{n}^{0}\text{-}\mathsf{RT}_{<\infty,1}^{1}$ denotes $\Sigma_{n}^{0}\text{-}\mathsf{RT}_{<\infty}^{1}$.
  
\item $\Sigma_{n}^{0}\text{-}\mathsf{RT}_{<\infty,\ell}^{1-}$ ($\Sigma_{n}^{0}\text{-}\mathsf{RT}_{<\infty,\ell}^{1}$-minus) is the assertion:
For every $k$, every total $\Sigma_{n}^{0}$-function $\varphi\colon\mathbb{N}\to\{0,1,\dots,k-1\}$ satisfies the following condition ($\ell$-homogeneity property):
\[ \exists s \subseteq \{0, 1, \dots ,k-1\} \, (|s|\le\ell \land \forall a \, \exists b \, [a < b \land \exists i \in s \, \varphi(b,i)] ). \]
In particular, $\Sigma_{n}^{0}\text{-}\mathsf{RT}_{<\infty,1}^{1-}$ denotes $\Sigma_{n}^{0}\text{-}\mathsf{RT}_{<\infty}^{1-}$.
\end{enumerate}

\end{definition}

Note that over $\mathsf{RCA}_0$, it holds that $\Sigma_{1}^{0}\text{-}\mathsf{RT}_{<\infty,\ell}^{1-}\leftrightarrow\Sigma_{1}^{0}\text{-}\mathsf{RT}_{<\infty,\ell}^{1}\leftrightarrow\mathsf{RT}_{<\infty,\ell}^{1}$ by $\Delta_{1}^{0}\text{-}\mathsf{CA}_0$ for any natural number $\ell \ge 1$.

\begin{remark}
For each natural number $n,\ell \ge 1$, the claim $\Sigma_{n}^{0}\text{-}\mathsf{RT}_{<\infty,\ell}^{1}$ asserts that an $\ell$-homogeneous set can be taken. On the other hand, $\Sigma_{n}^{0}\text{-}\mathsf{RT}_{<\infty,\ell}^{1-}$ merely asserts that the condition of $\ell$-homogeneity ($\ell$-homogeneity property) is satisfied. 
Consequently, $\Sigma_{n}^{0}\text{-}\mathsf{RT}_{<\infty,\ell}^{1-}$ is a weaker assertion than $\Sigma_{n}^{0}\text{-}\mathsf{RT}_{<\infty,\ell}^{1}$.
\end{remark}
We now introduce the main theorem.
\begin{theorem}\label{thm:main}
The following hold.

\begin{enumerate}[topsep=0.5em,itemsep=0.5em]
\item For any natural numbers $n, \ell \ge 1$,
\[ \mathsf{RCA}_0\vdash\mathsf{RT}_{<\infty,\ell}^{n}\rightarrow\Sigma_{n}^{0}\text{-}\mathsf{RT}_{<\infty,\ell}^{1}. \]

\item For any natural numbers $n, \ell \ge 1$,
\[
\mathsf{RCA}_0\vdash\mathsf{RT}_{<\infty,\ell}^{n}\rightarrow\mathsf{B}\Sigma_{n+1}^{0}.
\]
\end{enumerate}
\end{theorem}
While it is widely known that the above statement $(1)$ holds in $\omega$-models or within strong enough induction, here, it is demonstrated within $\mathsf{RCA}_0$, \textit{i.e.}, just with $\Sigma^{0}_{1}$-induction.   
We also note that we have not demonstrated that $\mathsf{RCA}_0+\mathsf{RT}_{<\infty,\ell}^{n}\vdash\Sigma_{n}^{0}\text{-}\mathsf{RT}_{<\infty}^{1}$.
By Theorem \ref{thm:main}, the following ensues.

\begin{corollary}\label{corollary}
For each natural number $\ell \ge 1$, the first-order part of $\mathsf{RCA}_0+\mathsf{RT}_{<\infty,\ell}^{2}$ is the same as $\mathsf{RCA}_0+\mathsf{B}\Sigma_{3}^{0}$.
In particular, for each natural number $\ell \ge 1$,
\[\mathsf{RCA}_0\not\vdash\mathsf{RT}_{2}^{2}\rightarrow\mathsf{RT}_{<\infty,\ell}^{2}.\]
\end{corollary}

\begin{proof}
By Theorem \ref{thm:main}, for any natural number $\ell \ge 1$, $\mathsf{RT}_{<\infty,\ell}^{2}$ derives $\mathsf{B}\Sigma_{3}^{0}$ over $\mathsf{RCA}_0$. Furthermore, by Slaman and the second author \cite{Slaman-Yokoyama}, the first-order part of $\mathsf{RCA}_0+\mathsf{RT}_{<\infty}^{2}$ is the same as $\mathsf{RCA}_0+\mathsf{B}\Sigma_{3}^{0}$. Therefore, for any natural numbers $\ell\ge 1$, the first-order part of $\mathsf{RCA}_0+\mathsf{RT}_{<\infty,\ell}^{2}$ is the same as $\mathsf{RCA}_0+\mathsf{B}\Sigma_{3}^{0}$.
Also, by Chong, Slaman and Yang \cite{Chong-Slaman-Yang}, $\mathsf{RCA}_0+\mathsf{RT}_{2}^{2}$ does not imply $\mathsf{I}\Sigma_{2}^{0}$.
Then, the latter statement immediately follows.
\end{proof}
Here, by Wang \cite[Corollary 3.4]{Wang}, it immediately follows that for any natural numbers $3\le n$ and $1\le\ell$, $\mathsf{RCA}_0\not\vdash\mathsf{RT}_{2}^{2}\rightarrow\mathsf{RT}_{<\infty,\ell}^{n}$. Corollary \ref{corollary} effectively shows an enhanced version of this.

Here, in relation to Theorem \ref{thm:main}, we prove the following theorem.
\begin{theorem}\label{thm:equivalence}
For any natural numbers $n,\ell \ge 1$,
\[\mathsf{RCA}_0\vdash\mathsf{B}\Sigma_{n+1}^{0}\leftrightarrow \Sigma_{n}^{0}\text{-}\mathsf{RT}_{<\infty}^{1-}\leftrightarrow\Sigma_{n}^{0}\text{-}\mathsf{RT}_{<\infty,\ell}^{1-}.\]
\end{theorem}

\begin{proof}[Proof of Theorem \ref{thm:equivalence}]
Here, recall that for any natural numbers $n \ge 1$, $\mathsf{B}\Pi_{n}^{0}$ is equivalent to $\mathsf{B}\Sigma_{n+1}^{0}$ over $\mathsf{RCA}_0$.
Then, by considering inductively, it suffices to show that for any natural numbers $n,\ell \ge 1$,
\[\mathsf{RCA}_0+\mathsf{B}\Sigma_{n}^{0}\vdash\mathsf{B}\Pi_{n}^{0}\leftrightarrow\Sigma_{n}^{0}\text{-}\mathsf{RT}_{<\infty}^{1-}\leftrightarrow\Sigma_{n}^{0}\text{-}\mathsf{RT}_{<\infty,\ell}^{1-}.
\]
The proof is carried out in two steps.

\begin{enumerate}[label=(\roman*),topsep=0.5em,itemsep=0.5em]
\item First, we prove the equivalence \[\mathsf{RCA}_0+\mathsf{B}\Sigma_{n}^{0}\vdash\mathsf{B}\Pi_{n}^{0}\leftrightarrow\Sigma_{n}^{0}\text{-}\mathsf{RT}_{<\infty}^{1-}.\]

For the backward direction $(\leftarrow)$, we proceed by contradiction. Assume $\Sigma_{n}^{0}\text{-}\mathsf{RT}_{<\infty}^{1-}\land\neg\mathsf{B}\Pi_{n}^{0}$. Then, by $\neg\mathsf{B}\Pi_{n}^{0}$, there exists a $\Pi_{n}^{0}$ formula $\forall z \, \theta(x,y,z)$ (where $\theta$ is a $\Sigma_{n-1}^{0}$ formula) and a witness $w$ satisfying the following two conditions:
\begin{equation}
\forall x < w \, \exists y \, \forall z \, \theta(x,y,z) \label{lem:equivalence-1}
\end{equation}

\begin{equation} \forall v \, \exists x < w \, \forall y < v \, \exists z \, \neg\theta(x,y,z) \label{lem:equivalence-2}
\end{equation}
Let $v$ be arbitrary. By \eqref{lem:equivalence-2}, there exists an $x_0<w$ that satisfies the following:
\[
\forall y < v \, \exists z \, \neg\theta(x_0,y,z)
\]
Over $\mathsf{RCA}_0$, $\neg\theta$ is equivalent to a $\Pi_{n-1}^{0}$ formula. Thus, by $\mathsf{B}\Sigma_{n}^{0}$, the following bounds exist:
\[
\exists m \, \forall y < v \, \exists z < m \, \neg\theta(x_0,y,z)
\]
Thus, we have:
\begin{equation} \exists x < w \, \exists m \, \forall y < v \, \exists z < m \, \neg\theta(x,y,z) \label{lem:equivalence-3}
\end{equation}

We can then define a function $c(v)$ by:

\begin{equation} c(v)=\min\{ \langle x,m \rangle \mid x < w \land \forall y < v \, \exists z < m \, \neg\theta(x,y,z) \}.
\end{equation}
By \eqref{lem:equivalence-3} and the least number principle for $\Pi_{n-1}^{0}$ formulas ($\mathsf{L}\Pi_{n-1}^{0}$), which follows from $\mathsf{B}\Sigma_{n}^{0}$, $c$ is a total function.

Moreover, since $c$ is defined via bounded quantification and logical connectives applied to $\Sigma_{n-1}^{0}$ and $\Pi_{n-1}^{0}$ formulas, the induction/bounding hypothesis $\mathsf{B}\Pi_{n-1}^{0}$ ensures that $c$ can be expressed by a $\Sigma_{n}^{0}$ formula. Therefore, $c$ is a total $\Sigma_{n}^{0}$-function.

Next, we define $d\colon\mathbb{N}\to\mathbb{N}$ by
\[
d(v):=(c(v))_0.
\]
Since $d$ is the composition of $c$ and a total computable function, $d$ is also a total $\Sigma_{n}^{0}$-function. Furthermore, by the definition of $c$, $(c(v))_0<w$ holds, meaning that $d$ is a total function with a bounded range. Thus, we can apply $\Sigma_{n}^{0}\text{-}\mathsf{RT}_{<\infty}^{1-}$ to $d$, which gives:
\begin{equation}   \exists x < w \, \forall a \, \exists b > a \, \forall y < b \, \exists z \, \neg\theta(x,y,z)
\end{equation}
Then,
\begin{equation}
\exists x < w \, \forall y \, \exists z \, \neg\theta(x,y,z)
\end{equation}
However, this contradicts \eqref{lem:equivalence-1}. Therefore, by contradiction, the backward direction holds.

Next, we show the forward direction $(\rightarrow)$. Assume $\mathsf{B}\Sigma_{n+1}^{0}$. Fix $k$ and let $\varphi$ be an arbitrary total $\Sigma_{n}^{0}$-function $\varphi\colon\mathbb{N}\to\{0,1,\dots,k-1\}$. We will show by contradiction that there exists a color that appears infinitely often. Suppose for contradiction that this does not hold, i.e.,
\[
\forall i < k \, \exists x \, \forall z \, (z > x \rightarrow \neg\varphi(z,i)).
\]
Since the subformula $\forall z \, (z > x \rightarrow \neg\varphi(z,i))$ is $\Pi_{n}^{0}$, $\mathsf{B}\Sigma_{n+1}^{0}$ implies:
\[
\exists v \, \forall i < k \, \exists x < v \, \forall z \, (z > x \rightarrow \neg\varphi(z,i)).
\]
In particular, this implies:
\[
\exists v \, \forall i < k \, \forall z > v-1 \, \neg\varphi(z,i).
\]
Then,
\[
\exists v \, \forall i < k \, \neg\varphi(v,i).
\]
However, this contradicts our assumption that the range of $\Sigma_{n}^{0}$-function $\varphi$ is contained in $\{0,1,\dots,k-1\}$. Thus, by contradiction, $\Sigma_{n}^{0}\text{-}\mathsf{RT}_{<\infty}^{1-}$ holds, which completes the proof of $\mathsf{RCA}_0+\mathsf{B}\Sigma_{n}^{0}\vdash\mathsf{B}\Pi_{n}^{0}\leftrightarrow\Sigma_{n}^{0}\text{-}\mathsf{RT}_{<\infty}^{1-}$.

\item Finally, we show that for all natural numbers $\ell\ge 1$,
\[ \mathsf{RCA}_0+\mathsf{B}\Sigma_{n}^{0}\vdash\Sigma_{n}^{0}\text{-}\mathsf{RT}_{<\infty,\ell}^{1-}\leftrightarrow\Sigma_{n}^{0}\text{-}\mathsf{RT}_{<\infty}^{1-}.
\]
Since the direction $(\leftarrow)$ is trivial, we focus on $(\rightarrow)$. Suppose $\Sigma_{n}^{0}\text{-}\mathsf{RT}_{<\infty,\ell}^{1-}$ holds. For any $k$ and total $\Sigma_{n}^{0}$-function $\varphi\colon\mathbb{N}\to\{0,1,\dots,k-1\}$, we have:

\begin{equation}
\exists s \subseteq \{0,1,\dots,k-1\} \, (|s|\le\ell \land \forall a \, \exists b \, [a < b \land \exists i \in s \, \varphi(b,i)]).
\end{equation}

We assume $\Sigma_{n}^{0}\text{-}\mathsf{RT}_{<\infty}^{1-}$ fails. Then, there exist $k$ and a total $\Sigma_{n}^{0}$-function $\varphi\colon\mathbb{N}\to\{0,1,\dots,k-1\}$ such that:

\begin{equation}
\forall i < k \, \exists a \, \forall b \, (a < b \rightarrow \neg\varphi(b,i)). \label{lem:equivalence-4}
\end{equation}

We fix those $\varphi$ and $\ell$ and let $s \subseteq \{0,1,\dots,k-1\}$ with $|s|\le\ell$ be the set obtained by applying $\Sigma_{n}^{0}\text{-}\mathsf{RT}_{<\infty,\ell}^{1-}$ for $\varphi$.
Then, for each $i \in s$, let $a_i$ be a witness satisfying the condition of \eqref{lem:equivalence-4}. Now, we define $m:=\max\{a_i \mid i\in s\}$. (Since the size of $s$ is at most $\ell$, the existence of $m$ can be demonstrated within $\mathsf{RCA}_0$.)
Then the following holds:

\begin{equation}
\forall m' > m \, \forall i \in s \, \neg\varphi(m',i). \label{lem:equivalence-5}
\end{equation}

On the other hand, by the assumption of $\Sigma_{n}^{0}\text{-}\mathsf{RT}_{<\infty,\ell}^{1-}$, there must exist a value such that:
\[
\exists m' > m \, \exists i \in s \, \varphi(m',i)
\]
which contradicts \eqref{lem:equivalence-5}. Thus, by contradiction, we obtain $(\rightarrow)$.
\end{enumerate}
\end{proof}

Here, to prove Theorem \ref{thm:main}, we employ the following two lemmas.

\begin{lemma}[Limit Lemma]\label{lem:limit}
For any natural number $n\ge 1$, over $\mathsf{RCA}_0+\mathsf{B}\Sigma_{n}^{0}$, for all $k$ and any total $\Sigma_{n}^{0}$-function $\varphi\colon\mathbb{N}\to\{0,1,\dots,k-1\}$, the following holds:
\begin{multline}
\exists g\colon[\mathbb{N}]^{n}\to\{0,1,\dots,k-1\} \, \forall x \, \exists i < k \,
[\exists s_1 > x \, \forall t_1 > s_1 \, \exists s_2 > t_1 \dots \\ \exists s_{n-1} > t_{n-2} \, \forall t_{n-1} > s_{n-1} \, g(x,t_1,\dots,t_{n-1})=i \land \varphi(x,i)].
\end{multline}
\end{lemma}

\begin{lemma}\label{lem:hom_eval}
For any natural number $n \ge 1$, let $k$ be a natural number, $\varphi$ be an arbitrary total $\Sigma_{n}^{0}$-function $\varphi\colon\mathbb{N}\to\{0,1,\dots,k-1\}$, $g\colon[\mathbb{N}]^{n}\to\{0,1,\dots,k-1\}$ be an approximation of $\varphi$ taken by Lemma \ref{lem:limit} and $H$ be an $\ell$-homogeneous set for $g$ with a color set $s \subseteq \{0,1,\dots,k-1\}$ of size $\ell$. In this case, $H$ is also an $\ell$-homogeneous set of $\varphi$.
\end{lemma}

First, we prove Theorem \ref{thm:main} using these lemmas.
\begin{proof}[Proof of Theorem \ref{thm:main}]
We proceed with the proof of both $(1)$ and $(2)$ simultaneously by induction on the standard natural number $n\ge 1$.

\begin{enumerate}[label=(\roman*),topsep=0.5em,itemsep=0.5em]
  \item \textbf{Case $n=1$:} By Theorem \ref{thm:equivalence}, for all natural numbers $\ell \ge 1$, we have
  \[ \mathsf{RCA}_0\vdash\mathsf{B}\Sigma_{2}^{0}\leftrightarrow\mathsf{RT}_{<\infty,\ell}^{1}\leftrightarrow\Sigma_{1}^{0}\text{-}\mathsf{RT}_{<\infty,\ell}^{1}.
  \]
  Therefore, the theorem holds for $n=1$.

  \item \textbf{Inductive step:} Assume that the theorem holds for a given $n \ge 1$; we will show that it holds for $n+1$.
  By the inductive hypothesis, $\mathsf{RCA}_0\vdash\mathsf{RT}_{<\infty,\ell}^{n}\rightarrow\mathsf{B}\Sigma_{n+1}^{0}$.
  Also, it easily follows that $\mathsf{RCA}_0\vdash\mathsf{RT}_{<\infty,\ell}^{n+1}\rightarrow\mathsf{RT}_{<\infty,\ell}^{n}$. Therefore, we can freely assume $\mathsf{B}\Sigma_{n+1}^{0}$ in our context.

  Furthermore, by Theorem \ref{thm:equivalence}, for each natural number $\ell \ge 1$,
  \[\mathsf{RCA}_0\vdash\Sigma_{n+1}^{0}\text{-}\mathsf{RT}_{<\infty,\ell}^{1-}\leftrightarrow\mathsf{B}\Sigma_{n+2}^{0}. \]
  Thus, regarding $(2)$, it suffices to show that
  \[\mathsf{RCA}_0\vdash\mathsf{RT}_{<\infty,\ell}^{n+1}\rightarrow\Sigma_{n+1}^{0}\text{-}\mathsf{RT}_{<\infty,\ell}^{1-}.
  \]
 (Furthermore, (1) is also shown naturally in the process of proving (2).) Now, let $k$ be a natural number and $\varphi(x,y)$ be an arbitrary total $\Sigma_{n+1}^{0}$-function $\varphi\colon\mathbb{N}\to\{0,1,\dots,k-1\}$. Then, by Lemma \ref{lem:limit}---whose validity can be verified using $\mathsf{B}\Sigma_{n+1}^{0}$---it follows that there exists a $g\colon[\mathbb{N}]^{n+1}\to\{0,1,\dots,k-1\}$ that is an approximation of $\varphi$ taken by Lemma \ref{lem:limit}.

  By assuming $\mathsf{RT}_{<\infty,\ell}^{n+1}$, we obtain an $\ell$-homogeneous set $H$ and a subset $s \subseteq \{0,1,\dots,k-1\}$ of size $\ell$. By Lemma \ref{lem:hom_eval}, it holds that
  \[
  \forall h \in H \, \exists i \in s \, \varphi(h,i).
  \]
  Therefore, we obtain $(1)$.
  Also, since $H$ is an infinite set, it follows that
  \[
   \forall x \, \exists h' > x \, \exists i \in s \, \varphi(h',i).
   \]
  Then, we have established $\Sigma_{n+1}^{0}\text{-}\mathsf{RT}_{<\infty,\ell}^{1-}$.
  Consequently,
  \[\mathsf{RCA}_0\vdash\mathsf{RT}_{<\infty,\ell}^{n+1}\rightarrow\mathsf{B}\Sigma_{n+2}^{0}.
  \]
  Therefore, we obtain $(2)$.
\end{enumerate}

By induction, we obtain both $(1)$ and $(2)$.
\end{proof}

Next, we prove the lemmas.

\begin{proof}[Proof of Lemma \ref{lem:limit}]

This can be proved in the same way as the standard limit lemma. However, from the perspective of reverse mathematics, note that the bounding principle is used in the course of the proof.

We proceed by induction on the standard natural number $n \ge 1$.
\begin{enumerate}[label=(\roman*),topsep=0.5em,itemsep=0.5em]
  \item \textbf{Case $n=1$:} When a $\Sigma_{1}^{0}$ formula satisfies the uniqueness, it is $\Delta_{1}^{0}$-definable. Then, by $\Delta_{1}^{0}\text{-}\mathsf{CA}_0$, $\varphi$ exists as a computable total function, and the lemma immediately follows for $n=1$. (In the case of $n=1$, this effectively constitutes a proof of the standard graph theorem.)

  \item \textbf{Inductive step:} Assume that the lemma holds for $n$; we will show it holds for $n+1$ under $\mathsf{B}\Sigma_{n+1}^0$.

  Let $\varphi(x,y)\equiv\exists a \, \forall b \, \theta(a,b,x,y)$ where $\theta$ is a $\Sigma_{n-1}^{0}$ formula.
  Then, we fix $k$ and assume that $\varphi$ is a total $\Sigma_{n+1}^{0}$-function $\varphi\colon\mathbb{N}\to\{0,1,\dots,k-1\}$.

  We then define $h\colon[\mathbb{N}]^2\to\mathbb{N}$ by $h(x,s):=\min\{\langle a,y \rangle \mid y<k \land \forall b < s \, \theta(a,b,x,y) \}$.

  Since this formula involves only bounded quantification and logical operations on $\Sigma_{n-1}^{0}$ formulas, it is defined by a $\Sigma_{n}^{0}$ formula by $\mathsf{B}\Sigma_{n}^{0}$. Furthermore, since a minimum value exists by $\mathsf{L}\Sigma_{n}^{0}$ (which follows from $\mathsf{B}\Sigma_{n+1}^{0}$), $h$ is a total function.

  Now, we fix $x_0$ arbitrarily. By assumption regarding $\varphi$, there exist $a'$ and $y'<k$ such that $\forall b \, \theta(a',b,x_0,y')$. Since $\forall b < s \, \theta(a',b,x_0,y')$ holds for any $s$, the definition of $h$ implies that $h(x_0,s)\le\langle a',y' \rangle$ holds for any $s>x_0$.

  Then, we define $h'\colon\mathbb{N}\to\mathbb{N}$ by $h'(e)=h(x_0,x_0+e+1)$. Here, $h'$ is a total $\Sigma_{n}^{0}$-function with a bounded range. Using Theorem \ref{thm:equivalence}, $\mathsf{B}\Sigma_{n+1}^{0} \equiv \Sigma_{n}^{0}\text{-}\mathsf{RT}_{<\infty}^{1-}$, we can show that $h'$ has a homogeneous value $i<k$. (Since $h'(e)$ is bounded and monotonically non-decreasing, the existence of a color that appears infinitely often implies that $h'$ becomes fixed at a constant value for sufficiently large values.)

  Then, since $h'$ is monotonically non-decreasing, there exists some $m$ such that for any $m' \ge m$, $h'(m')$ stabilizes at the value $\langle a_0,y_0 \rangle$.

  Therefore, by the definition of $h'$, it holds that:
  \begin{equation}
   \forall s \, \forall b < s \, \theta(a_0,b,x_0,y_0). \label{lem:limit-1}
   \end{equation}
  In particular, we obtain $\exists a \, \forall b \, \theta(a,b,x_0,y_0)$. By the uniqueness assumption on $\varphi$, it follows that $y_0=y'$.

  We define $f'\colon[\mathbb{N}]^2\to\{0,1,\dots,k-1\}$ by $f'(x,e)=(h(x,x+e))_1$. By the above argument, $f'$ is a total $\Sigma_{n}^{0}$-function with a bounded range, and it satisfies:
\begin{equation}
 \forall x \, \exists i < k \, [\exists s_1 > x \, \forall t_1 > s_1 \, f'(x,t_1)=i \land \varphi(x,i)]. \label{lem:limit-2}
 \end{equation}
  Next, we define $f\colon\mathbb{N}\to \{0,1,\dots,k-1\}$ by $f(x)=f'((x)_0,(x)_1)$. (Note that we choose a pairing function $\langle \cdot, \cdot \rangle$ whose first and second projections are monotonically non-decreasing, and $(x)_0$ and $(x)_1$ represent the first and second components of $x$, respectively.)

  Since this is a total $\Sigma_{n}^{0}$-function with a bounded range, by the inductive hypothesis, there exists a $g'\colon[\mathbb{N}]^{n}\to \{0,1,\dots,k-1\}$ which is an approximation of the formula "$f(x)=i$":
  \begin{multline} \forall x \, \big[\exists i < k \, \exists s_2 > x \, \forall t_2 > s_2 \dots \exists s_{n} > t_{n-1} \, \forall t_{n} > s_{n} \, \\
  g'(x,t_2,\dots,t_n)=i=f(x)\big]. \label{lem:limit-3}
  \end{multline}
  Then, by the definition of $f'$, it follows that:
  \begin{equation}
   \forall t \, \big[f(\langle x,t \rangle)=i \leftrightarrow f'(x,t)=i\big]. \label{lem:limit-4}
   \end{equation}
  For an arbitrary $x$, we can choose $s_1>x$ and $i<k$ satisfying \eqref{lem:limit-2}. Taking an arbitrary $t_1 > s_1$ and applying \eqref{lem:limit-3} and \eqref{lem:limit-4} to $\langle x,t_1 \rangle$, we have:
  \begin{multline}
\exists i < k \, \big[\exists s_2 > \langle x,t_1 \rangle \, \forall t_2 > s_2 \dots \exists s_{n} > t_{n-1} \, \forall t_{n} > s_{n}\\
g'(\langle x,t_1 \rangle,t_2,\dots,t_n)=f(\langle x,t_1 \rangle)=f'(x,t_1)=i\big]. \label{lem:limit-5}
  \end{multline}

  By the conditions of $t_1$, \eqref{lem:limit-2}, and \eqref{lem:limit-5}, it follows that:
  \begin{multline}
\forall x \, \exists i < k \, \big[\exists s_1 > x \, \forall t_1 > s_1 \, \exists s_2 > t_1 \dots \exists s_n > t_{n-1} \, \forall t_n > s_n \\
g'(\langle x,t_1 \rangle,t_2,\dots,t_n)=i \land \varphi(x,i) \big]. \label{lem:limit-6}
  \end{multline}
Now, we define $g\colon[\mathbb{N}]^{n+1}\to\{0,1,\dots,k-1\}$ as
\[
g(x_0,\dots,x_n):=g'(\langle x_0,x_1 \rangle, \langle x_0,x_1 \rangle +x_2, \dots ,\langle x_0,x_1 \rangle+x_2+ \dots +x_{n-1}+x_n).
\]
By \eqref{lem:limit-6}, this $g$ is the desired function.
 \end{enumerate}
\end{proof}

\begin{proof}[Proof of Lemma \ref{lem:hom_eval}]
Fix $k$ and a $\Sigma_{n}^{0}$-function $\varphi\colon\mathbb{N}\to\{0,1,\dots,k-1\}$, and let $g$ be the approximation of $\varphi$ taken by Lemma \ref{lem:limit} and let $H$ be an $\ell$-homogeneous set of $g$. Then, by definition of $g$, there exists $i<k$ such that:
\begin{multline}
\big[\exists s_1 > h_0 \, \forall t_1 > s_1 \dots \exists s_{n-1} > t_{n-2} \, \forall t_{n-1} > s_{n-1} \, \\
g(h_0,t_1,\dots,t_{n-1})=i\big] \land \varphi(h_0,i). \label{lem:hom_eval-1}
\end{multline}
We fix $i<k$ that satisfies the above.
By choosing such an $s_1>h_0$ satisfying \eqref{lem:hom_eval-1}, the following holds:
\begin{multline}
 \big[\forall t_1 > s_1 \dots \exists s_{n-1} > t_{n-2} \, \forall t_{n-1} > s_{n-1} \\
 g(h_0,t_1,\dots,t_{n-1})=i\big] \land \varphi(h_0,i). \label{lem:hom_eval-2}
 \end{multline}
Since $H$ is infinite, there exists $h_1\in H$ such that $s_1<h_1$. By \eqref{lem:hom_eval-2}, in particular, this $h_1$ satisfies:
\begin{multline}
\big[\exists s_2 > h_1 \, \forall t_2 > s_2 \dots \exists s_{n-1} > t_{n-2} \, \forall t_{n-1} > s_{n-1} \\
g(h_0,h_1,\dots,t_{n-1})=i\big] \land \varphi(h_0,i).
\end{multline}
We take $s_2$ satisfying the condition above. Since $H$ is infinite, we can find $h_2\in H$ such that $s_2<h_2$. Repeating this process $n-1$ times yields an $n$-tuple $\langle h_0,\dots,h_{n-1} \rangle\in[H]^{n}$ such that:
\[
 g(h_0,\dots,h_{n-1})=i \land \varphi(h_0,i).
 \]
Then, since $H$ is an $\ell$-homogeneous set of $g$, there is a set $s\subseteq\{0,1,\dots,k-1\}$ whose size is $\ell$ and $i \in s$.
By the definition of an $\ell$-homogeneous set, the above result holds for the same $s$, regardless of the choice of $h_0\in{H}$.

Therefore, we can conclude that $H$ is also an $\ell$-homogeneous set of $\varphi$.
\end{proof}

This completes the proofs of all the Theorems and Lemmas.
\section{Conclusions and future work}
In this paper, we discussed the relationship between mathematical induction and the thin set theorem. We showed that for a fixed $n$, the thin set theorem $\mathsf{RT}_{<\infty,\ell}^{n}$ allows for the derivation of $\mathsf{B}\Sigma_{n+1}^{0}$, independently of $\ell$. However, it is not yet known whether this is the optimal lower bound regarding induction. Specifically, the following open question remains:

\begin{question}
For any natural numbers $n \ge 1$ and $\ell \ge 1$ such that $d_{n-1}<\ell$, does the following hold?
\[ \mathsf{RCA}_0\not\vdash\mathsf{RT}_{<\infty,\ell}^{n}\rightarrow\mathsf{I}\Sigma_{n+1}^{0} \]
\end{question}
By Cholak and Patey \cite{Cholak-Patey}, when $3\le n$ and $\ell < d_{n-1}$, $\mathsf{RT}_{<\infty,\ell}^{n}$ is equivalent to $\mathsf{ACA}_0$, which easily implies $\mathsf{I}\Sigma_{n+1}^{0}$. For this reason, we restrict our attention to the case $d_{n-1}<\ell$.

For $n=1$, Theorem \ref{thm:equivalence} shows that $\mathsf{RCA}_0\vdash\mathsf{RT}_{<\infty,\ell}^{1}\leftrightarrow\mathsf{B}\Sigma_{2}^{0}$ for all natural numbers $\ell \ge 1$, which implies $\mathsf{RCA}_0\not\vdash\mathsf{RT}_{<\infty,\ell}^{1}\rightarrow\mathsf{I}\Sigma_{2}^{0}$. For $n=2$, Slaman and the second author \cite{Slaman-Yokoyama} found that $\mathsf{RCA}_0\not\vdash\mathsf{RT}_{<\infty,\ell}^{2}\rightarrow\mathsf{I}\Sigma_{3}^{0}$ for any natural number $\ell\ge 1$.

Therefore, the case $n \ge 3$ presents the most fundamental question. If the pattern observed for $n=1$ and $n=2$ extends to $n \ge 3$, this question would have a positive answer. The relationship between $\mathsf{RT}_{<\infty,\ell}^n$ and the hierarchy of induction principles is summarized in Figure \ref{fig:hierarchy}.

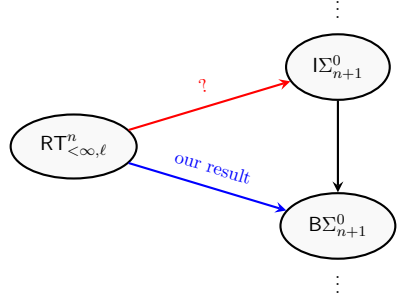
\begin{figure}[htbp]
\centering
\begin{tikzpicture}[
    scale=0.7, transform shape,
    >=stealth,
    circled/.style = {draw, ellipse, align=center, inner sep=6pt, font=\large, thick, fill=gray!5}
]
    \node[circled] (rt) at (0, 0) {$\mathsf{RT}_{<\infty,\ell}^n$};
    \node[circled] (isigma) at (5, 1.5) {$\mathsf{I}\Sigma_{n+1}^0$};
    \node[circled] (bsigma) at (5, -1.5) {$\mathsf{B}\Sigma_{n+1}^0$};
    
    \node (dots_top) at (5, 2.7) {$\vdots$};
    \node (dots_bot) at (5, -2.5) {$\vdots$};
    \draw[->, red, thick] (rt) -- node[above=4pt, red, sloped] {?} (isigma);
    \draw[->, blue, thick] (rt) -- node[above=4pt, blue, sloped] {our result} (bsigma);
    \draw[->, thick] (isigma) -- (bsigma);
    
\end{tikzpicture}

\caption{Inductive strength of $\mathsf{RT}_{<\infty,\ell}^n$ and the hierarchy of induction/bounding principles over $\mathsf{RCA}_0$.}
\label{fig:hierarchy}
\end{figure}

Furthermore, we propose another open question:
\begin{question}
 For any natural numbers $n, \ell \ge 1$, does there exist a natural number $m$ satisfying the following?
 \[ \mathsf{RCA}_0\vdash\mathsf{RT}_{\ell+1,\ell}^{n+1}\rightarrow\mathsf{RT}_{<\infty,m}^{n} \]
\end{question}
Note that for all natural numbers $n, \ell \ge 1$ and $\ell^{\prime}>\ell$, $\mathsf{RT}_{\ell^{\prime},\ell}^{n+1}$ is equivalent to $\mathsf{RT}_{\ell+1,\ell}^{n+1}$ over $\mathsf{RCA}_0$.
(Here, $\mathsf{RT}_{\ell^{\prime},\ell}^{n+1}$ asserts that every coloring $c\colon[\mathbb{N}]^{n+1}\rightarrow\{0,1,\dots,\ell^{\prime}-1\}$ admits an infinite set $H$ satisfying $|c([H]^{n+1}) |\le \ell$.)

If this question has an affirmative answer, then by Theorem \ref{thm:main}, we could conclude that for all natural numbers $\ell$ and $n$, $\mathsf{RCA}_0\vdash\mathsf{RT}_{\ell+1,\ell}^{n+1}\rightarrow\mathsf{B}\Sigma_{n+1}^{0}$.
If $\mathsf{RT}_{\ell+1,\ell}^{n+1}$ derives $\mathsf{ACA}_0$, this question is trivial. There are known examples of $n,\ell$ where $\mathsf{RT}_{\ell+1,\ell}^{n+1}$ is not equivalent to $\mathsf{ACA}_0$ and there exists an $m$ satisfying $\mathsf{RCA}_0\vdash\mathsf{RT}_{\ell+1,\ell}^{n+1}\rightarrow\mathsf{RT}_{<\infty,m}^{n}$. For example, Dorais et al.~\cite[Propositions 5.6 and 5.7]{Dorais-Dzhafarov-Hirst-Mileti-Shafer} found $\mathsf{RCA}_0+\mathsf{RT}_{3,2}^{2}\vdash\mathsf{RT}_{<\infty}^{1}$ and $\mathsf{RCA}_0+\mathsf{RT}_{4,3}^{3}\vdash\mathsf{RT}_{<\infty}^{2}$. Our future goal is to solve these questions.

\bibliographystyle{plain}
\bibliography{reference}
\end{document}